\documentclass{article}

\usepackage[utf8]{inputenc}
\usepackage{authblk}
\usepackage{setspace}
\usepackage[margin=1.25in]{geometry}
\usepackage{graphicx}
\graphicspath{ {./figures/} }
\usepackage{subcaption}
\usepackage{amsmath}
\usepackage{lineno}
\usepackage{amsfonts}
\usepackage{amssymb, amsthm, amsmath, amsfonts}
\usepackage{wasysym}
\usepackage{mathrsfs}
\usepackage{hyperref}
\usepackage{graphicx}
\usepackage{lineno}
\usepackage[colorinlistoftodos]{todonotes}
\usepackage{listings}
\usepackage{cancel, enumerate}
\usepackage{rotating, environ}
\usepackage{caption}
\usepackage{subcaption}
\usepackage[inline]{enumitem}
\usepackage{dirtree}
\usepackage{xcolor}
\usepackage{tipa}
\usepackage{float}
\usepackage{changepage}
\usepackage{tikz-cd}

\newtheorem{defn}{Definition}
\newtheorem{prop}{Proposition}
\newtheorem{lemma}{Lemma}
\newtheorem{cor}{Corollary}
\newtheorem{ex}{Example}

\newtheorem{rmk}{Remark}
\newtheorem{qstn}{Question}
\newtheorem{thm}{Theorem}

\renewcommand{\P}{\mathbb{P}}
\newcommand{\subseq}{\subseteq}

\newcommand{\Z}{\mathbb{Z}}

\newcommand{\R}{\mathbb{R}}

\newcommand{\ds}{\displaystyle}

\newcommand{\C}{\mathbb{C}}

\newcommand{\F}{\mathbb{F}}

\newcommand{\kar}{\text{char}}

\renewcommand{\P}{\mathbb{P}}

\newcommand{\wilde}{\widetilde}

\makeatletter
\newcommand{\tpitchfork}{%
  \vbox{
    \baselineskip\z@skip
    \lineskip-.52ex
    \lineskiplimit\maxdimen
    \m@th
    \ialign{##\crcr\hidewidth\smash{$-$}\hidewidth\crcr$\pitchfork$\crcr}
  }%
}
\makeatother

\newcommand{\Pcal}{\mathcal{P}}

\newcommand{\+}{\oplus}

\usepackage[style=numeric, 
citestyle=numeric-comp]{biblatex}
\title{The geproci property in positive characteristic}

\author{Jake Kettinger}

\affil{Department of Mathematics, University of Nebraska - Lincoln}

\date{}

\begin{document}

\maketitle

%%%%%% Abstract %%%%%%
\begin{abstract}
 The geproci property is a recent development in the world of geometry. We call a set of points $Z\subseq\P_k^3$ an $(a,b)$-geproci set (for GEneral PROjection is a Complete Intersection) if its projection from a general point $P$ to a plane is a complete intersection of curves of degrees $a\leq b$. Nondegenerate examples known as grids have been known since 2011. Nondegenerate nongrids were found starting in 2018, working in characteristic 0. Almost all of these new examples are of a special kind called half grids.
     
    Before the work in this paper-- based partly on the author's thesis-- only a few examples of geproci nontrivial non-grid non-half grids were known and there was no known way to generate more. Here, we use geometry in the positive characteristic setting to give new methods of producing geproci half grids and non-half grids.
\end{abstract}

%%%%%% Main Text %%%%%%

\section{Introduction}
While complete intersections have been a topic of much study for many years in algebraic geometry, the study of the geproci property has emerged relatively recently. Much of the groundwork in this study has been laid in the works \cite{CFFHMSS}, \cite{CM}, and \cite{PSS}, which will be cited often in this paper. We will begin with the definition of geproci (from: \textbf{ge}neral \textbf{pro}jection \textbf{c}omplete \textbf{i}ntersection).
\begin{defn}
\normalfont
Let $K$ be an algebraically closed field. A finite set $Z$ in $\P^n_K$ is \textit{geproci} $\left(\text{\textipa{\t{dZ}@"pro\t{tS}i}}\right)$ if the projection $\overline{Z}$ of $Z$ from a general point $P\in\P^n_K$ to a hyperplane $H$ is a complete intersection in $H\cong\P^{n-1}_K$.
\end{defn}

An easy but degenerate example of a geproci set in $\P^n$ is a complete intersection in a hyperplane $\P^{n-1}\cong H\subseq\P^n$. In this paper, we are specifically interested in geproci sets in $\P^3_K$. (No nondegenerate examples are known in $\P^n$, $n>3$.) In the three-dimensional setting, we will specify that a configuration $Z\subseq \P^3_K$ is $(a,b)$-geproci (where $a\leq b$) if the image of $Z$ under a general projection into $\P^2_K$ is the complete intersection of a degree $a$ curve and a degree $b$ curve. We will use the notation $\{a,b\}$-geproci in instances when we do not want to require $a\leq b$.

There are two easy-to-understand types of geproci sets. One type as noted above is any complete intersection in a plane: it will project from a general point isomorphically to another complete intersection in any other plane, and so is geproci. The other type is a grid, which we will now define.
\begin{defn}
\normalfont Given a curve $A\subseq\P^3$ comprising a finite set of $a$ pairwise-disjoint lines a curve $B\subseq \P^3$ comprising a finite set of $b$ pairwise-disjoint lines, such that every line in $A$ intersects every line in $B$ transversely, the $ab$ points of intersection form an $(a,b)$-\textbf{grid}. 
\end{defn}
The set of points $Z$ of an $(a,b)$-grid is $(a,b)$-geproci. The image $\overline{Z}$ of $Z$ under a general projection is equal to the intersection of the images $\overline{A}$ and $\overline{B}$ of $A$ and $B$, which are unions of $a$ lines in the plane and $b$ lines in the plane respectively, and thus $\overline{A}$ and $\overline{B}$ are curves of degrees $a$ and $b$, respectively, meeting at $ab$ points. Thus $\overline{Z}$ is a complete intersection.

These two types (sets of coplanar points and grids) are well understood, so are called \textit{trivial}. What is not yet well understood is how nontrivial geproci sets can arise. The existing work on the geproci property has been done over fields of characteristic 0. What is new with this paper are the results in characteristic $p>0$, starting in the second section. For the rest of this section we will only discuss work which has been done in characteristic 0.

The first nontrivial examples of geproci sets came from the root systems $D_4$ and $F_4$ \cite{CM} and so themselves are called $D_4$ and $F_4$. These are configurations in $\P^3$ containing $12$ points and 24 points, respectively \cite{HMNT}. It was also shown that $D_4$ is the \textit{smallest} nontrivial geproci set \cite{CM}, and the only nontrivial $(3,b)$-geproci set \cite{CFFHMSS}. (See Figure \ref{halfg} for the 12 points of $D_4$ and its 16 sets of 3 collinear points.)

The configurations $D_4$ and $F_4$ are examples of \textit{half grids}.
\begin{defn}\normalfont
A set $Z\subseq \P^3$ is a $\{\mu,\lambda\}$-\textbf{half grid} if $Z$ is a nontrivial $\{\mu,\lambda\}$-geproci set contained in a set of $\mu$ mutually-skew lines, with each line containing $\lambda$ points of $Z$.
\end{defn}
For example, the $D_4$ configuration is a ${4, 3}$-geproci half grid and can be covered by four mutually-skew lines, with each line containing three points, as Figure \ref{halfg} shows. The general projection of an $\{a,b\}$ half grid is a complete intersection of a union of $a$ lines and a degree $b$ curve that is not a union of lines. It is known that there is an $(a,b)$-half grid for each $4\leq a\leq b$ \cite{CFFHMSS}. No other infinite families of nontrivial geproci sets were known before the results in this paper, and only finitely many (indeed, three \cite{CFFHMSS}) non-half grid nontrivial geproci sets were known before the results in the next section. 

\begin{center}
\begin{figure}[H]
\centering
    \includegraphics[scale=0.7]{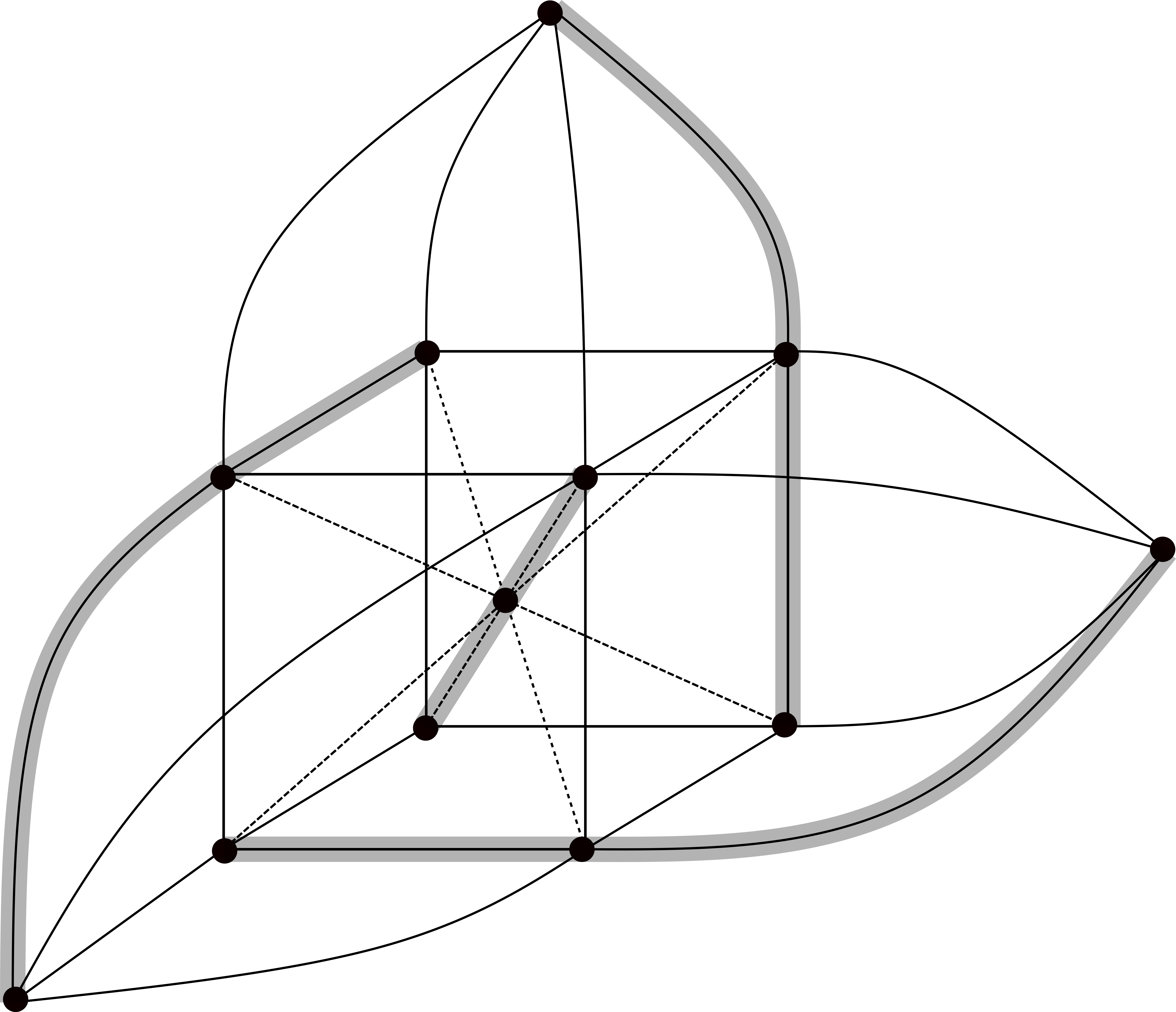}
    \caption{$D_4$ consists of 12 points arranged in 16 sets of 3 collinear points, and is covered by four skew lines as shown.}
    \label{halfg}
    \end{figure}
\end{center}

There seem to be strong links between geproci sets $Z$ and sets $Z$ admitting \textbf{unexpected cones} \cite{CFFHMSS, CM}.

\begin{defn}
\normalfont A finite set $Z\subseq \P^n_k$ admits an \textbf{unexpected} cone of degree $d$ when $$\dim[I(Z)\cap I(P)^d]_d>\max\left(0,\dim[I(Z)]_d-\binom{d+n-1}{n}\right)$$ for a general point $P\in\P^n_K$, where $I(Z)$ is the homogeneous ideal of $Z$ in $K[\P^n]$ and $[I(Z)]_d$ is its homogeneous component of degree $d$ \cite{HMNT, HMTG}. 
\end{defn}
This is said to be unexpected because one expects by a naive dimension count that the vector subspace of homogeneous polynomials in $[I(Z)]_d$ that are singular with multiplicity $d$ at a general point $P$ would have codimension $\ds\binom{n+d-1}{n}$ (since being singular at $P$ to order $d$ imposes $\ds\binom{n+d-1}{n}$ conditions on $[I(Z)]_d$). Therefore it is called unexpected when more such hypersurfaces exist than a naive dimension count would lead one to expect. Chiantini and Migliore showed that every $(a,b)$-grid with $3\leq a\leq b$ admits unexpected cones of degrees $a$ and $b$ \cite{CM}.

\section{The Geproci Property over Finite Fields}
\subsection{Spreads}
While examples of nontrivial geproci configurations (especially nontrivial non-half grids) have proven rather elusive in the characteristic 0 setting, we will see in this paper that they arise quite naturally over finite fields. In the finite field setting, we make generous use of the study of spreads over projective space, which we will define now.
\begin{defn}
\normalfont
Let $\P^{2t-1}_k$ be a projective space of odd dimension over a field $k$. Let $S$ be a set of $(t-1)$-dimensional linear subspaces of $\P^{2t-1}_k$, each of which is definedover $k$. We call $S$ a \textbf{spread} if each point of $\P^{2t-1}_k$ is contained in one and only one member of $S$.
\end{defn}
Over a finite field, spreads always exist for each $t\geq 1$ \cite{BB}. In our three-dimensional case, we have $t=2$. Therefore a spread in $\P^3_k$ will be a set of mutually-skew lines defined over $k$ that cover $\P^3_k$.
\begin{ex}\label{spread}
\normalfont
Here we show an example of a spread based on \cite{BB}. Given a field extension $k\subseq L$ with (as vector spaces) $\dim_k L=t$, we get a map $$\P_k^{2t-1}=\P_k(k^{2t})=\P_k(L^2)\longrightarrow\P_L(L^2)=\P^1_L$$ with linear fibers $\P_k(L)=\P(k^t)=\P_k^{t-1}$, giving a spread. When we take $k=\R$, $t=2$, and $L=\C$, we get $$\P^3_\R\longrightarrow\P^1_\C=S^2.$$ Composing with the antipodal map $S^3\to\P^3_\R$ gives the well-known Hopf fibration $S^3\to S^2$ with fibers $S^1$.
\end{ex}
\begin{ex}
\normalfont
Here we give another construction of spreads for $\P^3$ for fields of positive characteristic based on \cite{BB} and \cite{H}. Let $\F_q$ be a finite field of size $q$ and characteristic $p$, first where $p$ is an odd prime. Let $r\in\F_q$ be such that the polynomial $x^2-r\in\F_q[x]$ is irreducible; that is, $r$ has no square root in $\F_q$. Denote by $L_r(a,b)$ the line in $\P^3_{\F_q}$ through the points $(1,0,a,b)$ and $(0,1,rb,a)$. Denote by $L(\infty)$ the line through the points $(0,0,1,0)$ and $(0,0,0,1)$. Then the set of lines $$S_r=\{L_r(a,b),L(\infty):a,b\in\F_q\}$$ is a spread in $\P^3_{\F_q}$ (since $\P^3_{\F_q}$ has $(q+1)(q^2+1)=q^3+q^2+q+1$ points and one can check (using the fact that $r$ is not a square in $\F_q$) that the lines are skew, but there are $q^2+1$ lines and each line has $q+1$ points).

In the case $\kar\, \F_q=2$, we want to choose $r\in\F_q$ to be such that the polynomial $x^2+x+r$ is irreducible in $\F_q[x]$. Then define $L_r(a,b)$ to be the line in $\P^3_{\F_q}$ through the points $(1,0,a,b)$ and $(0,1,br,a+b)$. Then $S_r=\{L_r(a,b),L(\infty):a,b\in\F_q\}$ is a spread. 

\end{ex}
\begin{thm}\label{main}
Let $\F_q$ be the field of size $q$, where $q$ is some power of a prime. Then $Z=\P^3_{\F_q}\subseq \P^3_{\overline{\F}_q}$ is a $(q+1,q^2+1)$-geproci half grid.
\end{thm}
\begin{proof}
First we will show that there is a degree $(q+1)$ cone containing $Z$ having a singularity of multiplicity $q+1$ at a general point $P\in\P^3_{\overline{\F}_q}$. Let $P=(a,b,c,d)\in\P^3_{\overline{\F}_q}$. Let \begin{align*}
   M= \begin{pmatrix}a&    b&   c&   d\\
a^q&b^q&c^q&d^q\\
x&    y&    z&   w\\
x^q&y^q&z^q&w^q\\ \end{pmatrix}.
\end{align*}
Then we claim $F=\det M$
is such a cone.

First note that $F$ contains every point of $Z$, because $x^q=x$ for each $x\in\F_q$. Furthermore, the terms of $F$ can be combined into groups of 4 so that $F$ is the sum of terms of the form 
\begin{align*}
    (x^qyc^qd-x^qwc^qb)-(z^qya^qd-z^qwa^qb)=x^qc^q(yd-wb)-z^qa^q(yd-wb)\\
    =(x^qc^q-z^qa^q)(yd-wb)=(xc-za)^q(yd-wb)\in I^{q+1}((a,b,c,d))
\end{align*}
Thus $F$ is a cone $C_1$ of degree $q+1$ with vertex $(a,b,c,d)$ of multiplicity $q+1$.

 Now we will show there is a degree $q^2+1$ cone $C_2$ containing $Z$ having a general point $P$ of multiplicity $q^2+1$. By Example \ref{spread}, the space $\P^3_{\F_q}$ admits a spread of $q^2+1$ mutually-skew lines that covers all of $\P^3_{\F_q}$. Each line together with a fixed general point $P$ determines a plane. The union of the planes gives $C_2$.
 
 Projecting the $q^2+1$ lines from a general point $P\in\P^3_{\overline{\F}_q}$ to a general plane $\Pi=\P^2_{\overline{\F}_q}$ yields a set of $q^2+1$ lines in $\P^2_{\F_q}$ containing the $(q+1)(q^2+1)$ points of the image of $Z$.
 
  Now we will show that $C_1$ and $C_2$ do not have components in common; to this end, we will show that $C_1$ contains no line in $\P^3_{\overline{\F}_q}$ defined over $\P^3_{\F_q}$. Note that $C_1$ vanishes on such a line if and only if $F=0$, where $F=\det M$ and $$M=\begin{pmatrix}a&b&c&d\\ a^q&b^q&c^q&d^q\\ X&Y&Z&W\\X^q&Y^q&Z^q&W^q\end{pmatrix}$$ for $X=\eta_0u+\mu_0v$, $Y=\eta_1u+\mu_1v$, $Z=\eta_2u+\mu_2v$, and $W=\eta_3u+\mu_3v$ for all $(u,v)\in\P^1_{\F_q}$ where $(\eta_0,\eta_1,\eta_2,\eta_3)$ and $(\mu_0,\mu_1,\mu_2,\mu_3)$ are points on the line. If $r_1$, $r_2$, $r_3$, and $r_4$ are the rows of a $4\times 4$ matrix, we will denote the determinant of that matrix by $|r_1,r_2,r_3,r_4|$. In particular, taking the $r_i$ to be the rows of $M$, we have $F=|r_1,r_2,r_3,r_4|=|r_1,r_2,\eta u+\mu v,\eta u^q+\mu v^q|=0$ for all $(u,v)$.

  Since determinants are multilinear, we have \begin{align*}
  &|r_1,r_2,\eta u+\mu v,\eta u^q+\mu v^q|\\=&|r_1,r_2,\eta u,\eta u^q|+|r_1,r_2,\eta u,\mu v^q|+|r_1,r_2,\mu v,\eta u^q|+|r_1,r_2,\mu v,\mu v^q|\\=&|r_1,r_2,\eta u,\eta u|u^{q-1}+|r_1,r_2,\eta u,\mu v^q|+|r_1,r_2,\mu v,\eta u^q|+|r_1,r_2,\mu v,\mu v|v^{q-1}\\=&|r_1,r_2,\eta u,\mu v^q|+|r_1,r_2,\mu v,\eta u^q|=|r_1,r_2,\eta,\mu|uv^q+|r_1,r_2,\mu,\eta|u^qv\\=&|r_1,r_2,\eta,\mu|uv^q-|r_1,r_2,\eta,\mu|u^qv=|r_1,r_2,\eta,\mu|(v^{q-1}-u^{q-1})uv.
\end{align*} 
But $v^{q-1}-u^{q-1}\neq 0$ unless $u=v=0$ or $u/v\in\F_q$. Therefore $F$ is 0 for all $(u,v)$ only if $|r_1,r_2,\eta,\mu|=0$. By an appropriate choice of coordinates we get $\eta=(1,0,0,0)$, $\mu=(0,1,0,0)$, $r_1=(a',b',c',d')$, and $r_2=(a'^q,b'^q,c'^q,d'^q)$ for some point $(a',b',c',d')$ which is general since $(a,b,c,d)$ is general. Since $|r_1,r_2,\eta,\mu|$ is nonzero for $a'=b'=0$, $c'=1$, $d'\in\overline{\F}_q\setminus\F_q$, we see $|r_1,r_2,\eta,\mu|\neq 0$ for general $(a',b',c',d')$. We conclude that $C_1$ does not contain a line of $\P^3_{\overline{\F}_q}$ defined over $\P^3_{\F_q}$, and so $C_1$ has no components in common with $C_2$. (In fact, since $C_1$ contains the $q+1$ points of each line of $\P^3_{\overline{\F}_q}$ defined over $\P^3_{\F_q}$ but does not contain the line, $C_1$ meets each line of $\P^3_{\overline{\F}_q}$ defined over $\P^3_{\F_q}$ transversely.) Thus $C_1\cap C_2$ is a curve of degree $(q+1)(q^2+1)$ and contains the $(q+1)(q^2+1)$ lines through $P$ and points of $Z$, hence $C_1\cap C_2$ is exactly this set of lines.

So $\overline{Z}$ is a set of $(q+1)(q^2+1)$ points, which is the intersection of the curves $C_1\cap\Pi$ (of degree $q+1$) and $C_2\cap\Pi$ (of degree $q^2+1$), so $\overline{Z}$ is a $(q+1,q^2+1)$-complete intersection. Thus $Z$ is $(q+1,q^2+1)$-geproci.

\end{proof}
Furthermore, the degree $q+1$ and $q^2+1$ cones in the above proof are unexpected. We will show this with the help of the following lemma.
\begin{lemma}\label{thingy}
     Let $Z=\P^n_{\F_q}$ in variables $x_0,\dots,x_n$. Then $\dim[I(Z)]_{q+1}=1+2+\cdots+n=\ds\binom{n+1}{2}$.
\end{lemma}
\begin{proof}
    We will induct on $n$, starting with $n=1$. The product $$x_0(x_0-x_1)(x_0-2x_1)\cdots(x_0-(q-1)x_1)x_1$$ is the unique $q+1$ form (up to scalar multiplication) vanishing on all points of $Z$. So $\dim[I(Z)]_{q+1}=1$.

    Now let $n>1$ and $Z'=V(x_n)\subseq Z=\P^n_{\F_q}$, so we can regard $Z'$ as $Z'=\P^{n-1}_{\F_q}$. We can regard each element $f\in[I(Z')]_{q+1}$ as a form in the variables $x_0,\dots,x_{n-1}$ and thus defined over $\P^n$. Using this, we can define a map $\rho:[I(Z)]_{q+1}\to[I(Z')]_{q+1}$ by $\rho(f(x_0,\dots,x_{n-1},x_n))=f(x_0,\dots,x_{n-1},0)$. We can see that $\rho$ is surjective because each element $g\in[I(Z')]_{q+1}$ defines a cone over $Z'$ with vertex $v=(0,\dots,0,1)\in\P^n$. Thus $g$ vanishes at every line through $v$ and a point of $Z'$. But every point of $Z$ is on such a line, so $g\in[I(Z)]_{q+1}$, and thus $\rho$ is surjective.
    
    Now let $Y$ be the complement of $Z'$ in $Z$. Then we have $x_n[I(Y)]_{q}\subseq [I(Z)]_{q+1}$. Furthermore, for all $f\in [I(Z)]_{q+1}$, we see that $\rho(f)=0$ if and only if $f=0$ or $f=x_n\cdot h$ for some degree $q$ polynomial $h$ vanishing on $Y$. Hence $x_n[I(Y)]_{q}=\ker \rho$. This gives us the short exact sequence 
    \begin{center}
        \begin{tikzcd}
            0\arrow{r}&x_n[I(Y)]_{q}\arrow{r} & \left[I(Z)\right]_{q+1}\arrow{r} & \left[I(Z')\right]_{q+1}\arrow{r}&0\\
        \end{tikzcd}
    \end{center}
    where $\dim [I(Z')]_{q+1}=1+\cdots+(n-1)$ by the induction hypothesis. Now we must show that $\dim x_n[I(Y)]_q=n$. But $\dim x_n[I(Y)]_q=\dim[I(Y)]_q$ and $Y$ is a complete intersection of $n$ forms of degree $q$. For example, we can cut out $Y$ by the $n$ forms given by $$x_i(x_i-x_n)(x_i-2x_n)\cdots(x_i-(q-1)x_n)$$ for $0\leq i\leq n-1$. Hence $\dim[I(Y)]_q=n$, and so $\dim[I(Z)]_{q+1}=\dim[I(Z')]_{q+1}+\dim[I(Y)]_q=1+\cdots+n$.
\end{proof}
\begin{prop}
    The degree $q+1$ cone and degree $q^2+1$ cone in the proof of Theorem \ref{main} are unexpected. 
\end{prop}
\begin{proof}
From Lemma \ref{thingy}, we see that $\dim[I(Z)]_{q+1}=6$. In particular, $[I(Z)]_{q+1}$ is generated by the $2\times 2$ minors of the matrix $$\begin{pmatrix}
        x&y&z&w\\x^q&y^q&z^q&w^q\\
    \end{pmatrix}.$$  Since $6-\ds\binom{q+3}{3}<0$ for $q\geq 2$, and $\dim[I(Z)\cap I(P)^{q+1}]_{q+1}\geq 1>0$, we have that the above $q+1$ cone is indeed unexpected.

    To show that the degree $q^2+1$ cone is unexpected, we will first show that the $(q^2+1)(q+1)$ points of $\P^3_{\F_q}$ impose independent conditions on forms of degree $q^2+1$. We will show that for each $Q\in\P^3_{\F_q}$ that there is a degree $q^2+1$ form vanishing at every point $\P^3_{\F_q}$ except $Q$. Without loss of generality, we will take $Q=(0,0,0,1)$.

    We will start with the case $q\neq 2$. Then the union of planes given by the product $$\pi_x=\prod_{i=0}^{q-1}(w-ix)$$ contains every point of $\P^3_{\F_q}$ except those on the affine plane $\{(0,*,*,1)\}$. Similarly, the products $$\pi_y=\prod_{i=0}^{q-1}(w-iy)\text{ and }\pi_z=\prod_{i=0}^{q-1}(w-iz)$$ vanish everywhere except on the affine planes $\{(*,0,*,1)\}$ and $\{(*,*,0,1)\}$, respectively. Therefore, the product $\pi_x\pi_y\pi_z$ vanishes everywhere on $\P^3_{\F_q}$ except the point $(0,0,0,1)$. Since $\deg\pi_x\pi_y\pi_z=3q$, taking $\pi=w^{q^2-3q+1}\pi_x\pi_y\pi_z$ gives us a degree $q^2+1$ form vanishing at every point of $\P^3_{\F_q}$ except $Q$. Note that since $q>2$, $q^2-3q+1>0$, so $\pi$ is well-defined.

    Since the points of $Z=\P^3_{\F_q}$ impose independent conditions on the $q^2+1$ forms, we have $$\dim[I(Z)]_{q^2+1}=\binom{q^2+4}{3}-(q^2+1)(q+1).$$ Using our degree $q+1$ cone from the proof of Theorem \ref{main} as $F$, we have $$F[I(P)^{q^2-q}]_{q^2-q}\subseq [I(Z)+I(P)^{q^2+1}]_{q^2+1},$$ giving us $$\dim[I(P)^{q^2-q}]_{q^2-q}\leq\dim[I(Z)+I(P)^{q^2+1}]_{q^2+1}.$$ We know that $\dim[I(P)^{q^2-q}]_{q^2-q}=\ds\binom{q^2-q+3}{3}-\binom{q^2-q+2}{3}=\binom{q^2-q+2}{2}$, so in order to show the degree $q^2+1$ cone is unexpected it is sufficient to see that the following inequality holds: $$\binom{q^2-q+2}{2}>\binom{q^2+4}{3}-(q^2+1)(q+1)-\binom{q^2+3}{3}.$$ This inequality holds for $q\geq 3$. Thus for all prime powers $q\geq 3$, the degree $q^2+1$ cone in the proof of Theorem \ref{main} is unexpected.

     Now for the case $q=2$: First we wish to show that the fifteen points of $Z=\P^3_{\F_2}$ impose independent conditions on the quintic forms. Again taking $Q=(0,0,0,1)$ without loss of generality, we can take $\pi=w^2(w+x)(w+y)(w+z)$ as our degree 5 form vanishing at every point of $\P^3_{\F_2}$ except $Q$. Therefore the points indeed impose independent conditions. Thus $$\dim[I(Z)]_5=\binom{5+3}{3}-15=41$$ and so $\dim[I(Z)]_5-\ds\binom{5+2}{3}=41-35=6.$ A computation in Macaulay2 reveals that $$\dim[I(Z)+I(P)^5]_5=7>6,$$ thus the degree $q^2+1$ cone from the proof of Theorem \ref{main} is unexpected for $q=2$ as well. 
     \end{proof}
     The Macaulay2 commands used to show $\dim[I(Z)+I(P)^5]_5=7$ are as follows.

\vspace{\baselineskip}
     
     \texttt{i1:K=frac(ZZ/2[a,b,c,Degrees=>\{0,0,0\}]);}

    \texttt{i2:R=K[x,y,z,w];}

    \texttt{i3:V=\{\{0,0,0,1\},\{0,0,1,0\},\{0,0,1,1\},\{0,1,0,0\},\{0,1,0,1\},\\
    \{0,1,1,0\},\{0,1,1,1\},\{1,0,0,0\},\{1,0,0,1\},\{1,0,1,0\},\\
    \{1,0,1,1\},\{1,1,0,0\},\{1,1,0,1\},\{1,1,1,0\},\{1,1,1,1\}\};}

\texttt{i4:IV=\{\};}

\texttt{i5:for i from 0 to \#V-1 do \{A=trim ideal(V\_i\_0*y-V\_i\_1*x,\\ V\_i\_0*z-V\_i\_2*x,V\_i\_0*w-V\_i\_3*x,V\_i\_1*z-V\_i\_2*y,V\_i\_1*w-V\_i\_3*y,\\ V\_i\_2*w-V\_i\_3*z);IV=IV|\{A\}\};}

\texttt{i6:I=intersect(IV);}

\texttt{i7:P=ideal(x-a*w,y-b*w,z-c*w);}

\texttt{i8:J=intersect(I,P\^{}5);}

\texttt{i9:M=mingens(J);}

\texttt{i10:for i from 0 to numgens(source(M))-1 do print degree(M\_i)}

\subsection{Maximal Partial Spreads}

Of particular interest to the hunt for geproci sets is the existence of \textit{maximal partial spreads}.
\begin{defn}
\normalfont    A \textbf{partial spread} of $\P^3_{\F_q}$ with \textbf{deficiency} $d$ is a set of $q^2+1-d$ mutually-skew lines of $\P^3_{\F_q}$. A \textbf{maximal partial spread} is a partial spread of positive deficiency that is not contained in any larger partial spread. We will denote the set of points of $\P^3_{\F_q}$ contained in the lines in a spread $S$ by $\Pcal(S)$.
    \end{defn}

Maximal partial spreads allow us to construct examples of many geproci sets as subsets of $\P^3_{\F_q}$, using the following corollary.

\begin{cor}
Let $S$ be a partial spread of $s$ lines in $\P^3_{\F_q}$. Then the set of points $\Pcal(S)\subseq\P^3_{\F_q}$ is $\{s,q+1\}$-geproci.
\end{cor}
\begin{proof}
The same degree $q+1$ cone $C_1$ from the proof of Theorem \ref{main} works in this case. The degree $s$ cone is the join of the $s$ lines with the general point $P$. It follows from the proof of Theorem \ref{main} that $C_1$ meets every line of $\P^3_{\F_q}$ transversely and thus that $\Pcal(S)$ is geproci.
\end{proof}

\begin{lemma}\label{lemma1}
Let $Z$ be an $\{a,b\}$-geproci set and let $Z'\subseq Z$ be a $\{c,b\}$-geproci subset, whose general projection shares with the general projection of $Z$ a minimal generator of degree $b$. Then the residual set $Z''=Z\setminus Z'$ is $\{a-c,b\}$-geproci.
\end{lemma}
\begin{proof}
This is Lemma 4.5 of \cite{CFFHMSS}, and the proof still works in positive characteristic.
\end{proof}

 \begin{thm}
    The complement $Z\subseq\P^3_{\F_q}$ of a maximal partial spread of deficiency $d$ is a nontrivial $\{q+1,d\}$-geproci set. Furthermore, when $d>q+1$, $Z$ is also not a half grid.
    \end{thm}
    
    \begin{proof}
    The first sentence of the Theorem comes directly from Corollary 1 and Lemma 1, except for being nontrivial. To demonstrate that $Z$ is nontrivial, suppose $Z$ is contained in a plane $H$. Let $Z'$ be the complement of $Z$. Then $Z'$ consists of $q+1$ points on $q^2+1-d$ lines. At most one of those lines can be in $H$, but each of the lines meet $H$. Thus $Z'$ has at least $q^2+1-d$ points in $H$, so $Z$ consists of at most $q^2+q+1-(q^2+1-d)=q+d$ points. This is impossible since $|Z|=(q+1)d>q+d$.

    Now suppose that $Z$ is a grid. Thus it consists of $q+1$ points on each of $d$ lines. But $Z'$ comes from a maximal partial spread, so $Z$ contains no set of $q+1$ collinear points. Thus $Z$ cannot be a grid, so $Z$ is nontrivial.
    
    Now we will prove that $Z$ is a nontrivial non-half grid if $d>q+1$. Recall that every line in $\P^3_{\F_q}$ consists of $q+1$ points. If $Z$ were a half grid, then either it contains subsets of $d$ collinear points or subsets of $q+1$ collinear points, but $d>q+1$, so the latter would be true. But we know from the above that $Z$ contains no subset of $q+1$ collinear points.
    
    %Then $Z$ would be contained in $q+1$ mutually-skew lines, with each line containing $d$ points of $Z$ (we know that it cannot be the other way around because then $Z$ couldn't be the complement of a maximal partial spread). This cannot happen when $d>q+1$, because each line contains only $q+1$ points.
    \end{proof}
    \subsection{Examples}
    
    \begin{ex}
    \normalfont
    By \cite{H}, if $q\geq 7$ and $q$ is odd, then $\P^3_{\F_q}$ has a maximal partial spread of size $n$ for each integer $n$ in the interval $\ds\frac{q^2+1}{2}+6\leq n\leq q^2-q+2$. In terms of deficiency $d=q^2+1-n$, we get the inequalities $q-1\leq d\leq\ds\frac{q^2+1}{2}-6$. Thus for every odd prime power $q\geq 7$ there is a maximal partial spread in $\P^3_{\F_q}$ of deficiency $d>q+1$ and thus a nontrivial non-half grid $(q+1,d)$-geproci set.
    \end{ex}
    \begin{rmk}
    \normalfont
    In addition to Heden's bounds \cite{H} showing the existence of maximal partial spreads, Mesner has provided a lower bound for the size of the deficiency $d$ at $\sqrt{q}+1\leq d$ \cite{M}. Glynn has provided an upper bound for $d$ at $d\leq (q-1)^2$ \cite{G}.
    \end{rmk}
     \begin{ex}
    \normalfont
    By Lemma \ref{lemma1}, for any line $L\subseq\P^3_{\F_2}$, the set $Z=\P^3_{\F_2}\setminus L$ is a $(3,4)$-geproci half grid. In fact, $Z$ has the same combinatorics as $D_4$, shown in Figure \ref{fig:d42} (that is, $Z$ consists of 12 points, each of which is on 4 lines, with each line containing 3 of the points). Specifically, in Figure \ref{fig:d42} we see $\P^3_{\F_2}\setminus V(x+y+z,w)$.
    \end{ex}
    \begin{center}
    \begin{figure}[H]
    \centering
        \includegraphics[scale=0.6]{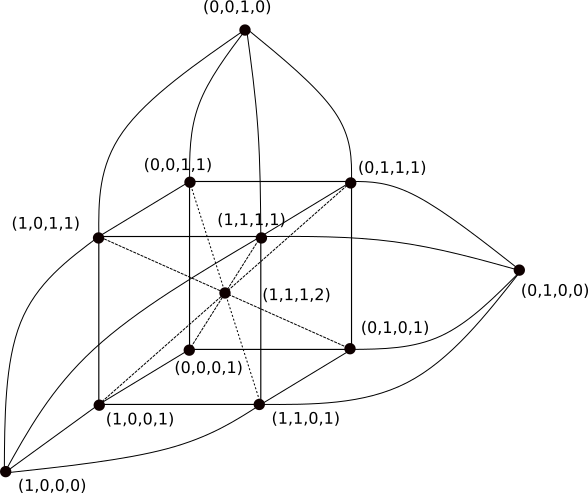}
        \caption{A $D_4$ in any characteristic.}
        \label{fig:d42}
        \end{figure}
    \end{center}
    \begin{ex}
    \normalfont
    There is (up to projective equivalence) a unique maximal partial spread in $\P^3_{\F_3}$ \cite{S}. This spread contains seven lines (as opposed to a complete spread, which contains ten). The complement $Z$ of the points of the maximal partial spread is a set of $12$ points in $\P^3_{\F_3}$ that is $(3,4)$-geproci and nontrivial. Furthermore, $Z$ has the same combinatorics as the $D_4$ configuration (that is, $Z$ is a set of 12 points, each of which is on 4 lines, with each line containing 3 of the points). Note that $Z$ is then a half grid, as shown in Figure \ref{fig:d42}. Specifically, Figure \ref{fig:d42} exhibits the points of $\P^3_{\F_3}$ in the complement of the maximal partial spread given by the seven lines $V(x+y,y+z+w)$, $V(x-y-z,y+w)$, $V(x-y+w,y+z)$, $V(x+y+z,w)$, $V(x-y+z, z+w)$, $V(x+y-z,x+w)$, and $V(x+z,x+y+w)$.
    \end{ex}

    \begin{ex}
    \normalfont
    There are (up to projective equivalence) fifteen maximal partial spreads in $\P^3_{\Z/7\Z}$ of size 45 and invariant under a group of order 5 (as opposed to a complete spread, which contains 50 lines) \cite{S}. Let $Z$ be the complement of the set of points of any of these maximal partial spreads. Then $Z$ is a set of 40 points that is a nontrivial $(5,8)$-geproci non-half grid. Furthermore, $Z$ has the same combinatorics as the Penrose configuration of 40 points \cite{CFFHMSS}.
    
    Note that if we look at two non-isomorphic maximal partial spreads $M$ and $M'$, and consider their complements $Z$ and $Z'$, then $Z$ and $Z'$ are non-isomorphic nontrivial non-half grid $(5,8)$-geproci sets. In fact, some such sets have stabilizers of different sizes! Of the fifteen up to isomorphism, there are nine with stabilizers of size 10, there is one with a stabilizer of size 20, there is one with a stabilizer of size 60, and there are four with stabilizers of size 120.

    An example of such a geproci set is\begin{align*}
    \{(0,0,1,3),(0,1,3,3),(0,1,3,5),(0,1,4,6),\\
(0,1,6,5),(1,0,1,3),(1,0,2,6),(1,0,4,5),\\
(1,0,4,6),(1,1,0,1),(1,1,0,4),(1,1,1,4),\\
(1,1,5,2),(1,2,1,6),(1,2,3,3),(1,2,5,2),\\
(1,2,6,5),(1,3,2,1),(1,3,4,4),(1,3,5,2),\\
(1,3,6,0),(1,4,0,5),(1,4,2,4),(1,4,4,1),\\
(1,4,6,2),(1,5,0,4),(1,5,1,0),(1,5,2,0),\\
(1,5,3,0),(1,5,3,1),(1,5,3,3),(1,5,3,6),\\
(1,5,4,5),(1,5,5,0),(1,5,5,2),(1,5,6,3),\\
(1,6,0,3),(1,6,1,5),(1,6,2,1),(1,6,6,6)\}.\\
    \end{align*}
This example is the complement of a maximal partial spread of size 45 with a stabilizer of size 60.
    
    We also used Macaulay2 to check that at least one configuration of each size stabilizer is Gorenstein. This contrasts with the case in characteristic 0, where only one nontrivial Gorenstein geproci set is known, up to projective equivalence: the Penrose configuration. \cite{CFFHMSS}

    One can determine this using the following commands in Macaulay2 with the example set of points from above.

    \texttt{i1:K=toField(ZZ/7[a,b,c,d]);}

    \texttt{i2:R=K[x,y,z,w];}

    \texttt{i3:V=\{\{0,0,1,3\},\{0,1,3,3\},\{0,1,3,5\},\{0,1,4,6\},\\
\{0,1,6,5\},\{1,0,1,3\},\{1,0,2,6\},\{1,0,4,5\},\\
\{1,0,4,6\},\{1,1,0,1\},\{1,1,0,4\},\{1,1,1,4\},\\
\{1,1,5,2\},\{1,2,1,6\},\{1,2,3,3\},\{1,2,5,2\},\\
\{1,2,6,5\},\{1,3,2,1\},\{1,3,4,4\},\{1,3,5,2\},\\
\{1,3,6,0\},\{1,4,0,5\},\{1,4,2,4\},\{1,4,4,1\},\\
\{1,4,6,2\},\{1,5,0,4\},\{1,5,1,0\},\{1,5,2,0\},\\
\{1,5,3,0\},\{1,5,3,1\},\{1,5,3,3\},\{1,5,3,6\},\\
\{1,5,4,5\},\{1,5,5,0\},\{1,5,5,2\},\{1,5,6,3\},\\
\{1,6,0,3\},\{1,6,1,5\},\{1,6,2,1\},\{1,6,6,6\}\};}

\texttt{i4:IV=\{\};}

\texttt{i5:for i from 0 to \#V-1 do \{A=trim ideal(V\_i\_0*y-V\_i\_1*x,\\ V\_i\_0*z-V\_i\_2*x,V\_i\_0*w-V\_i\_3*x,V\_i\_1*z-V\_i\_2*y,V\_i\_1*w-V\_i\_3*y,\\ V\_i\_2*w-V\_i\_3*z);IV=IV|\{A\}\};}

\texttt{i6:I=intersect(IV);}

\texttt{i7:betti res I}

\texttt{o7:}\begin{tabular}{ccccc}
     & 0& 1 & 2 & 3 \\
     total:&1&5&5&1\\
     0:&1&$\cdot$&$\cdot$&$\cdot$\\
     1:&$\cdot$&$\cdot$&$\cdot$&$\cdot$\\
     2:&$\cdot$&$\cdot$&$\cdot$&$\cdot$\\
     3:&$\cdot$&5&$\cdot$&$\cdot$\\
     4:&$\cdot$&$\cdot$&5&$\cdot$\\
     5:&$\cdot$&$\cdot$&$\cdot$&$\cdot$\\
     6:&$\cdot$&$\cdot$&$\cdot$&$\cdot$\\
     7:&$\cdot$&$\cdot$&$\cdot$&1\\
\end{tabular}

We can see from the Betti table that this set of points is Gorenstein. A similar calculation works to show the other geproci sets are Gorenstein.
    \end{ex}
    
    This pattern leads us to the following question: 
    \begin{qstn}
    Given the complement of a maximal partial spread $Z\subseq\P^3_{\F_q}$, when does $Z$ correspond to a nontrivial geproci set that exists in $\P^3_\C$? That is, when does there exist a nontrivial geproci set in $\P_\C^3$ that has the same combinatorics as $Z$?
    \end{qstn}
    
    \section{The Geproci Property with Infinitely-Near Points}
We can also consider configurations of points that include \textit{infinitely-near points}.

\begin{defn}
\normalfont
Let $A$ be a smooth point on an algebraic variety $X$. Let $\text{Bl}_A(X)$ denote the \textbf{blowup} of $X$ at $A$. Then a point $B\in\text{Bl}_P(X)$ is \textbf{infinitely-near} $A$ if $\pi_A(B)=A$ where $\pi_A:\text{Bl}_A(X)\to X$ is the standard blowup map.

On the other hand, if and $\pi_A(B)\neq A$, then $B$ and $A$ are \textbf{distinct}.
\end{defn}

Intuitively, $B$ corresponds to the direction of a line through $A$. In the plane, we can consider how a point $A$ and a point $B$ that is infinitely-near $A$ can uniquely determine a line, the same way a line can be uniquely determined by two distinct points. This is akin to determining a line from a point and a slope. In $\P^3$, we will consider how infinitely-near points impose conditions on forms the same way distinct points can.

We can extend the definition of geproci to include configurations with infinitely-near points by realizing $Z$ as a non-reduced 0-dimensional subscheme of $\P^3$. For example, let $A\in\P^3$ be a point and $L$ a line through $A$. Let $B$ be the point infinitely near $A$ corresponding to $L$. Then $I(\{A,B\})=I(L)+I(A)^2$ and the ideal of the image $\{\overline{A},\overline{B}\}$ of $\{A,B\}$ under projection from a point $P\notin L$ is $I(\overline{L})+I(\overline{A})^2$, where $\overline{L}$ is the image of $L$. A scheme $Z$ including infinitely near points is geproci if the projection $\overline{Z}$ of $Z$ from a general point $P$ to a plane is a complete intersection as a subscheme of $\P^2$.

In the following sets of points in $\P^3_{\F_2}$, we will denote a point $A$ together with a point infinitely-near $A$ as $A\times 2$. We will then specify what line the infinitely-near point corresponds to.

\begin{ex}\label{quasi}
\normalfont
We will consider the set of nine (not distinct) points in $\P^3_{K}$ where $\kar\,K=2$: $$Z=\{(1,0,0,0)\times 2, (0,1,0,0)\times 2,(0,0,1,0)\times 2, (0,0,0,1)\times 2,(1,1,1,1)\}$$ by choosing infinitely-near points for each of $(1,0,0,0)$, $(0,1,0,0)$, $(0,0,1,0)$, and $(0,0,0,1)$ to be the point that corresponds to the (respective) direction of the line through the given point and the point $(1,1,1,1)$.

The projection $\overline{Z}$ of these 9 points to the plane $w=0$ from a general point takes $(0,0,1)$, $(0,1,0)$, $(1,0,0)$ to themselves and $(1,1,1,1)$ and $(0,0,0,1)$ to general points. After a change of coordinates we can map the image of $(1,1,1,1)$ to $(1,1,1)$ and the image of $(0,0,0,1)$ to $(a,b,c)$. We will denote $$Z'=\{(0,0,1)\times 2,(0,1,0)\times 2,(1,0,0)\times 2,(a,b,c)\times 2,(1,1,1)\},$$ where the tangent directions of each point of multiplicity 2 correspond to the line connecting the point with $(1,1,1)$. Then $Z'$ is the base locus of a specific type of pencil of cubics called a quasi-elliptic fibration. Specifically, the quasi-elliptic pencil given by $Z$ has Dynkin diagram $\wilde{A}_1^{\+8}$. One can read more about the connection between Dynkin diagrams and (quasi-)elliptic fibrations in e.g. Cossec and Dolgachev \cite{CD}.

We can see that the conic $C_1=V(xy+xz+yz)$ contains the points $(0,0,1)$, $(0,1,0)$, and $(1,0,0)$, and the tangent lines of the three points all meet $(1,1,1)$. Additionally, the line $L_1$ connecting $(a,b,c)$ and $(1,1,1)$ has the appropriate slope to contain the remaining infinitely-near point. Therefore the cubic given by $C_1\cup L_1$ contains $Z'$.

Similarly, we can also construct a conic $C_2=V(cxy+bxz+ayz+(a+b+c)y^2)$ that contains the points $(0,0,1)$, $(0,1,0)$, $(a,b,c)$, and their respective infinitely-near points. Letting $L_2$ denote the line connecting $(1,0,0)$ and $(1,1,1)$, we get another cubic $C_2\cup L_2$ containing $Z'$. The two cubics share no components in common, and so $Z'$ is a complete intersection of two cubics.

Since $Z'$ is projectively equivalent to $\overline{Z}$, we get $\overline{Z}$ is a complete intersection. Therefore $Z$ is $(3,3)$-geproci. Note that $Z$ is a nontrivial non-half grid. What makes this work is the fact that the tangent lines of a conic in characteristic 2 are concurrent.
\end{ex}
\begin{rmk}
    \normalfont We can also see that Example \ref{quasi} provides examples of unexpected cones. Letting $\{\alpha_0,\alpha_1,\alpha_2,\alpha_3\}=\{x,y,z,w\}$, we can construct a (non-minimal) generating set for $I(Z)$ as $$\mathcal{A}=\{\alpha_i\alpha_j(\alpha_k+\alpha_\ell):i,j\neq k,\,\,i,j\neq\ell,\,\,k\neq\ell\}.$$ (Note that this set includes both the polynomials where $i=j$ and $i\neq j$.) A computation in Macaulay2 reveals that the ideal generated by $\mathcal{A}$ can be minimally generated by 11 cubic polynomials. Therefore $\dim [I(Z)]_3=11$. We also have $\ds\binom{3+2}{3}=10$, so $\dim[I(Z)]_3-\ds\binom{3+2}{3}=1$.

    But we also know that $\dim[I(Z)+I(P)^{3}]_3\geq 2$ by for example taking the join of the two planar cubics making up the complete intersection of $Z'$ with the vertex $P$. Therefore we have the inequality $\dim[I(Z)+I(P)^{3}]_3>\dim[I(Z)]_3-\ds\binom{3+2}{3}>0$ and so the cubic cones are indeed unexpected.
\end{rmk}
\vspace{\baselineskip}

\vspace{\baselineskip}
\begin{ex}
\normalfont
Let $\kar\,K=2$. Now consider the 6 points $$Z=\{(1,0,0,0)\times 2, (0,1,0,0)\times 2,(0,0,1,0)\times 2\},$$ where the infinitely near point for each is in the direction of $(0,0,0,1)$. We will show that this is $(2,3)$-geproci.

\begin{proof}

First we will look at the following scheme of points in $\P^2$: $$Z'=\{(1,0,0)\times 2, (0,1,0)\times 2,(0,0,1)\times 2\}$$ where the infinitely-near point for each is in the direction of $(1,1,1)$. We will show that this set of 6 points is a complete intersection of a conic and a cubic, and then show that a general projection of $Z$ onto any plane is projectively equivalent to $Z'$. Note that $Z'$ is contained in the conic $A=V(xy+xz+yz)$ and the cubic $B=V((x+y)(x+z)(y+z))$. Also note that $A$ and $B$, have no components in common, since $A$ is an irreducible conic and $B$ is the union of three lines. Therefore $Z'$ is a complete intersection of a conic and a cubic.

Now let us return to $Z\subseq\P^3$. Let us project $Z$ from a general point $P\in \P^3$ onto a general plane $\Pi\subseq \P^3$. Since the lines corresponding to each infinitely-near point meet at $(0,0,0,1)$, and since projection from a point preserves lines (and therefore the intersection of lines), the images of the three infinitely-near points under the projection $\pi_{P,\Pi}$ will also correspond to three concurrent lines. In other words, $Z$ will map to the set $$\pi_{P,\Pi}(Y)=\{\pi_{P,\Pi}(1,0,0,0)\times 2,\pi_{P,\Pi}(0,1,0,0)\times 2,\pi_{P,\Pi}(0,0,1,0)\times 2\}$$ where each infinitely-near point is in the direction of $\pi_{P,\Pi}(0,0,0,1)$. For a general point $P$, the images of the three ordinary points in $Z$ and the point $\pi_{P,\Pi}(0,0,0,1)$ will not be collinear. Therefore we can map $\Pi$ to $\P^2$ and use an automorphism of the plane to map $\pi_{P,\Pi}(1,0,0,0)$ to $(1,0,0)$, $\pi_{P,\Pi}(0,1,0,0)$ to $(0,1,0)$, $\pi_{P,\Pi}(0,0,1,0)$ to $(0,0,1)$, and $\pi_{P,\Pi}(0,0,0,1)$ to $(1,1,1)$. Then we are in the same situation as $Z'$, which is a complete intersection of a conic and a cubic.

Note that $Z$ is a half grid, since the cubic containing $Z$ is a union of three lines, but the conic is irreducible.

The unique quadric cone containing $Z$ with a vertex at $(a,b,c,d)$ is given by $cdxy+bdxz+adyz+abw^2$.
\end{proof}
\end{ex}
\vspace{\baselineskip}
\begin{ex}
\normalfont
Let $\kar\,K=2$. Now consider the 9 points $$Z=\{(1,0,0,0)\times 2, (1,1,0,0)\times 2, (0,1,0,0)\times 2, (0,0,1,0)\times 2, (0,0,0,1)\},$$ by choosing as our infinitely-near points for $(1,0,0,0)$, $(1,1,0,0)$, $(0,1,0,0)$, and $(0,0,1,0)$ the points that correspond to the respective directions to the point $(0,0,0,1)$. First we will look at the following set of points in $\P^2_{K}$: $$Z'=\{(1,0,0)\times 2,(a,0,1)\times 2,(0,0,1)\times 2,(1,1,1)\times 2,(0,1,0)\}$$ where $a\neq 0$ and each infinitely-near point is in the direction of $(0,1,0)$. These nine points are a complete intersection of $(y^2+xz)(x+az)$ and $y^2(x+z)$. Since every set of four points, no three of which are collinear, maps can be mapped to every other such set of four points by a linear automorphism, every projection of $Z$ onto any plane $\Pi$ will be isomorphic to the configuration $Z'$ for some $a\in K\setminus\{1,0\}$, and so $Z$ is a nontrivial $(3,3)$-geproci set.
\end{ex}
\vspace{\baselineskip}
The preceding example is particularly interesting because the general projection of $X$ is not only a $(3,3)$ complete intersection, but as in Example \ref{quasi}  it is also the set of base points of a quasi-elliptic fibration (specifically one with Dynkin diagram $\wilde{A}_1^{\+4}\+\wilde{D}_4$).

%%%%% Citations in the text %%%%%%

\printbibliography

\end{document}